\documentclass[12pt,reqno]{amsart}
\usepackage{amsmath, amsthm, amssymb, bm, eufrak,enumitem}
\usepackage[noadjust]{cite}

\advance \topmargin by -\headheight
\advance \topmargin by -\headsep
     
\evensidemargin \oddsidemargin
\newtheorem{theorem}{Theorem}[section]
\newtheorem{lemma}[theorem]{Lemma}
\newtheorem{corollary}[theorem]{Corollary}

\newtheorem{hypothesis}[theorem]{Hypothesis}

\theoremstyle{definition}

\theoremstyle{remark}

\numberwithin{equation}{section}

\title{On the Asymptotic Formula in Waring's Problem with Shifts}
\author{Kirsti Biggs}
\address{School of Mathematics, University of Bristol, University Walk, Clifton, Bristol, BS8 1TW, United Kingdom}
\email{kirsti.biggs@bristol.ac.uk}
\thanks{The author is supported by an EPSRC Doctoral Training Partnership}

\begin{document}

\begin{abstract}
\vspace{-1em}
We show that for integers $k\geq 4$ and $s\geq k^2+(3k-1)/4$, we have an asymptotic formula for the number of solutions, in positive integers $x_i$, to the inequality $\left|(x_1-\theta_1)^k+\dotsc+(x_s-\theta_s)^k-\tau\right|<\eta$, where $\theta_i\in(0,1)$ with $\theta_1$ irrational, $\eta\in(0,1]$, and $\tau>0$ is sufficiently large. We use Freeman's variant of the Davenport--Heilbronn method, along with a new estimate on the Hardy--Littlewood minor arcs, to obtain this improvement on the original result of Chow.
\vspace{-2.2em}
\end{abstract}

\maketitle

\section{Introduction} \label{intro}
In its classical form, Waring's problem asks whether every positive integer $N$ can be represented as a sum of $s$ $k$th powers of integers, where $s$ does not depend on $N$. One generalisation of this problem, studied by Davenport and Heilbronn in the 1940s (see, for example, \cite{davheil}), was to consider diagonal inequalities of the form
\begin{equation} \label{DHineq}
\left|\lambda_1 x_1^k+\dotsc+\lambda_s x_s^k\right|<1,
\end{equation}
where the coefficients are non-zero, not all of the same sign, and not all in rational ratio. In particular, they proved that for a suitable sequence $(P_n)_{n=1}^\infty$ of large real numbers, with $P_n\rightarrow\infty$ as $n\rightarrow\infty$, the number of integer solutions to (\ref{DHineq}) with $1\leq x_1,\dotsc,x_s\leq P_n$ is at least $cP_n^{s-k}$, for some constant $c>0$. In \cite{freeman}, Freeman developed a version of their method which works for \emph{all} large values of $P$. This has become known as Freeman's variant of the Davenport--Heilbronn method, and is now a crucial tool in the study of Diophantine inequalities.

In \cite{chowWP}, Chow introduced and studied a different analogue of Waring's problem, namely that of approximating real numbers by $k$th powers of shifted integers. More precisely, for a large, positive real number $\tau$, we are interested in counting integer solutions to the inequality 
\begin{equation} \label{ineq}
\left|(x_1-\theta_1)^k+\dotsc+(x_s-\theta_s)^k-\tau\right|<\eta,
\end{equation}
for fixed natural numbers $s\geq k\geq 2$, shifts $\theta_1,\dotsc,\theta_s \in (0,1)$ with $\theta_1$ irrational, and $0<\eta\leq1$. Let $N(\tau)=N_{s,k,\bm{\theta},\eta}(\tau)$ be the number of solutions to (\ref{ineq}) in positive integers $x_1,\dotsc,x_s$. In this paper, we reduce the minimum number of variables required to obtain an asymptotic formula for $N(\tau)$. To that end, let $s_0(k)=k^2+(3k-1)/4$. Our main result is the following:
\begin{theorem} \label{mainThm}
Let $k\geq 4$, and let $s\geq s_0(k)$. Then
\begin{equation} \label{asymp}
N(\tau)=2\eta\Gamma(1+1/k)^s\Gamma(s/k)^{-1}\tau^{s/k-1}+o(\tau^{s/k-1}). 
\end{equation}
\end{theorem}
Note that there is no explicit dependence on the shifts $\theta_1,\dotsc,\theta_s$ in the main term of (\ref{asymp}).

The best previously known bound for this problem is due to Chow, who showed in \cite{chowWP} that the asymptotic formula (\ref{asymp}) holds for $s\geq2k^2-2k+3$. However, an examination of the arguments underlying Chow's work reveals that the recent proof in \cite{BDG} of the Main Conjecture in Vinogradov's Mean Value Theorem, by Bourgain, Demeter and Guth, allows this constraint to be improved to $s\geq k^2+k+1$. Although our method also works for $k=3$, it does not improve on the best known value of 11 variables, also due to Chow, in \cite{cubes}.

To prove our result, we approximate the number of solutions to (\ref{ineq}) by a certain integral over the real line (see Section \ref{notation} for details). We use a dissection of the real line into major, minor and trivial arcs, as is usual in the Davenport--Heilbronn method, to evaluate this integral. However, in order to achieve our reduction in the number of variables required, we must also divide our arcs into points with or without good approximations by rationals with small denominators, commonly known as the major and minor arcs in the Hardy--Littlewood method.

The new estimate given in Section \ref{Aminorarc} extends the method of Wooley in \cite{asympform} to a setting appropriate to Diophantine inequalities. We first obtain a bound for the contribution to a certain mean value from points without good rational approximations, making use of the aforementioned result of Bourgain, Demeter and Guth. In order to give a more precise statement of our result, we must introduce some notation. As is usual in this area, we use $e(z)$ to denote $\exp(2\pi iz)$. For real numbers $P$, $\theta$ and $\alpha$, with $P$ large and $\theta\in(0,1)$, we define
\begin{equation*}
f_\theta(\alpha)=\sum_{1\leq x\leq P}e(\alpha(x-\theta)^k).
\end{equation*}
We define $\mathfrak{v}$ to be the real analogue of the classical Hardy--Littlewood minor arcs: namely, with $Q$ a real parameter satisfying $1\leq Q\leq P$, we define $\mathfrak{v}=\mathfrak{v}_Q$ to be the set
\begin{equation} \label{vdef}
\{\alpha\in\mathbb{R}: \mbox{for }a\in\mathbb{Z}\mbox{ and }q\in\mathbb{N}\mbox{ coprime},\left|q\alpha-a\right|\leq QP^{-k}\implies q>Q\}.
\end{equation}
Finally, we define the kernel function $K(\alpha)=\big(\frac{\sin(\pi\alpha)}{\pi\alpha}\big)^2$, which has the property (see~\cite[Lemma 4]{davheil}) that for any real number $t$, one has
\begin{equation*}
\int_{\mathbb{R}}e(t\alpha)K(\alpha)d\alpha=\max\{0,1-\left|t\right|\}.
\end{equation*}
We are now in a position to state the following result.
\begin{theorem} \label{minorsimp}
For natural numbers $s\geq 2$, $k\geq 2$, and for $\theta\in(0,1)$, we have \begin{equation} \label{msimp}
\int_{\mathfrak{v}} \left|f_\theta(\alpha)\right|^{2s}K(\alpha)\,d\alpha\ll  P^{\epsilon}Q^{-1}(P^{s+\frac{1}{2}k(k-1)}+P^{2s-k}).
\end{equation}
\end{theorem}

This can be viewed as an analogue of the bound
\begin{equation*}
\int_{\mathfrak{v}_{P/(2k)}\cap [0,1)} \left|f_0(\alpha)\right|^{2s}\,d\alpha\ll  P^{\epsilon-1}(P^{s+\frac{1}{2}k(k-1)}+P^{2s-k}),
\end{equation*}
which is \cite[Theorem 1.3]{asympform}.

We make use of a variant of Theorem \ref{minorsimp} (see Theorem \ref{minorarcs}, and the subsequent conclusion in Corollary \ref{minormixed}), which provides a key input to our application of Freeman's variant of the Davenport--Heilbronn method. The number of variables required to achieve this estimate is smaller than that required by Chow to bound the contribution from points on the minor and trivial arcs, and this enables us to make our improvement as stated in Theorem \ref{mainThm}. On the major arc, we use Chow's result to obtain the main term in the asymptotic formula, while on the remainder of the minor and trivial arcs, we show that the contribution is negligible. In order to do this, we make use of the measure of the set of points with good rational approximations, noting that these points constitute only a small fraction of any given unit interval.

Experts will recognise that there is the potential to apply the key ideas of this paper to related problems, such as that of counting integral solutions to Diophantine inequalities of the shape
\begin{equation*}
\left|\lambda_1(x_1-\theta_1)^k+\dotsc+\lambda_s(x_s-\theta_s)^k-\tau\right|<\eta,
\end{equation*}
where the $\lambda_i$ are real numbers, these being essentially a combination of (\ref{DHineq}) and (\ref{ineq}). We defer such considerations to a future occasion.

We now present a brief outline of the structure of the remainder of this paper. In Section \ref{notation}, we introduce the preliminary notation required throughout the paper. In Section \ref{Aminorarc}, we present our new estimate for the contribution from the classical Hardy--Littlewood minor arcs, which ultimately allows us to improve on previously known lower bounds for the number of variables required for the asymptotic formula to hold. In Section \ref{restofminor}, we show that negligible contributions are obtained from the remainder of the minor and trivial arcs not covered by Corollary \ref{minormixed}. In Section \ref{majorarc} we present a result of Chow on the major arc, giving the main term in the asymptotic formula for the number of solutions, thus completing the proof of Theorem \ref{mainThm}.

The author would like to thank Trevor Wooley for his supervision and for suggesting this line of research, and Sam Chow for helpful conversations.

\section{Preliminary notation} \label{notation}

We now introduce the conventions and pieces of standard notation which will be used in this paper. When a statement involves $\epsilon$, we mean that the statement holds for any suitably small value of $\epsilon>0$. We let $\bm{\theta}=(\theta_1,\dotsc,\theta_s)$, and use the vector notation $1\leq\bm{x}\leq P$ to mean that $1\leq x_i\leq P$ for all $i$. Throughout, we assume that $\tau$ is sufficiently large in terms of $s,k,\bm{\theta}$ and $\eta$.

Let $P=\tau^{1/k}$, and let $N^*(\tau)$ be the number of solutions to (\ref{ineq}) with $1\leq\bm{x}\leq P$. A solution which does not meet this condition can have at most one of the variables larger than $\tau^{1/k}$, and in this situation the remaining variables must each be at most some constant multiple of $\tau^{(k-1)/k^2}$. Thus, since we may assume that $s>k^2-k+1$, it follows that
\begin{equation*}
N(\tau)-N^*(\tau)\ll \tau^{(s-1)(k-1)/k^2} =o(\tau^{s/k-1}).
\end{equation*}
It therefore suffices to prove that
\begin{equation*}
N^*(\tau)=2\eta\Gamma(1+1/k)^s\Gamma(s/k)^{-1}\tau^{s/k-1}+o(\tau^{s/k-1}). 
\end{equation*}

We use the Davenport--Heilbronn kernel $K(\alpha;\eta)=\eta\bigg({\dfrac{\sin(\pi\eta\alpha)}{\pi\eta\alpha}}\bigg)^2$, which has the property (via a slight adaptation of \cite[Lemma 20.1]{davenport}) that for any real number $t$, one has
\begin{equation} \label{Kbound}
\int_{\mathbb{R}}e(t\alpha)K(\alpha;\eta)\,d\alpha=\max\{0,1-\left|t/\eta\right|\}.
\end{equation}

Consequently, letting
\begin{equation*}
f_\theta(\alpha)=\sum_{1\leq x\leq P}e(\alpha(x-\theta)^k),
\end{equation*}
and
\begin{equation*}
f_{\bm{\theta}}(\alpha)=f_{\theta_1}(\alpha)\cdots f_{\theta_s}(\alpha),
\end{equation*}
we observe that the integral
\begin{equation} \label{orthog}
\int_{\mathbb{R}}f_{\bm{\theta}}(\alpha) e(-\tau\alpha)K(\alpha;\eta)\,d\alpha
\end{equation}
provides a weighted count of the number of solutions to (\ref{ineq}). To be precise, a tuple $(x_1,\dotsc,x_s)$ contributes 1 whenever the left-hand side of (\ref{ineq}) is equal to zero, and $1-\zeta/\eta$ whenever the left-hand side of (\ref{ineq}) is equal to $\zeta$, for some $\zeta\in(0,\eta)$.

The following lemma demonstrates the existence of a certain positive function which provides a bound on the values of the exponential sums we are interested in.
\begin{lemma} \label{2.2}
Let $k\geq 2$ be an integer, and let $\xi,\theta_1,\theta_2\in(0,1)$ with $\theta_1$ irrational. Then there exists a positive real-valued function $T(P)$, for which $T(P)\rightarrow\infty$ as $P\to\infty$, such that
\begin{equation} \label{TPbound}
\sup_{P^{\xi-k}\leq \left|\alpha\right|\leq T(P)}\left|f_{\theta_1}(\alpha)f_{\theta_2}(\alpha)\right| \ll P^2 T(P)^{-1}.
\end{equation}
\end{lemma}
\begin{proof}
This is a special case of \cite[Lemma 2.2]{chowWP}.
\end{proof}

We divide up the real line into major, minor and trivial arcs, as is usual in the Davenport--Heilbronn method. We fix a real number $\xi\in (0,1)$, and apply Lemma \ref{2.2} to obtain the function $T(P)$. We then define
\begin{gather*}
\mathfrak{M}=\{\alpha\in\mathbb{R}:\left|\alpha\right|< P^{\xi-k}\},\\
\mathfrak{m}=\{\alpha\in\mathbb{R}:P^{\xi-k}\leq\left|\alpha\right|\leq T(P)\},
\end{gather*}
and
\begin{equation*}
\mathfrak{t}=\{\alpha\in\mathbb{R}:\left|\alpha\right|> T(P)\}.
\end{equation*}
We can therefore evaluate the integral (\ref{orthog}) using the dissection
\begin{equation} \label{DHdissect}
\mathbb{R}=\mathfrak{M}\cup\mathfrak{m}\cup\mathfrak{t}.
\end{equation}
The major arc provides the main term in the asymptotic formula for the number of solutions, while the minor and trivial arcs provide negligible contributions which form the error term.

In order to successfully evaluate the contributions from the central major arc (as in \cite[Section 3]{chowWP}), we must use a different kernel function related to $K(\alpha;\eta)$ to reduce the length of the interval which provides a non-negligible contribution. We define
\begin{equation} \label{LPdefn}
L(P)=\min\{\log{T(P)},\log{P}\},\quad\delta=\eta L(P)^{-1},
\end{equation}
and the upper and lower kernel functions
\begin{equation*}
K_\pm(\alpha)={\dfrac{\sin(\pi\alpha\delta)\sin(\pi\alpha(2\eta\pm\delta))}{\pi^2\alpha^2\delta}}.
\end{equation*}

These kernel functions are the same as those obtained in \cite[Lemma 1]{freeman} (applied with $a=\eta-\delta$ and $b=\eta$ for $K_-(\alpha)$, and with $a=\eta$ and $b=\eta+\delta$ for $K_+(\alpha)$, along with $h=1$ in both cases). Letting $U_c(t)$ denote the indicator function of the interval $(-c,c)$, the conclusion of that lemma gives us the bounds
\begin{equation*}
U_{\eta-\delta}(t)\leq \int_{\mathbb{R}}e(\alpha t)K_-(\alpha)d\alpha\leq U_{\eta}(t),
\end{equation*}
\begin{equation*}
U_{\eta}(t)\leq \int_{\mathbb{R}}e(\alpha t)K_+(\alpha)d\alpha\leq U_{\eta+\delta}(t),
\end{equation*}
and
\begin{equation} \label{K+-bound}
K_\pm(\alpha)\ll \min\{1,\alpha^{-2}L(P)\}.
\end{equation}

Letting
\begin{equation*}
R_\pm(P)=\int_{\mathbb{R}}f_{\bm{\theta}}(\alpha) e(-\tau\alpha)K_\pm(\alpha)\,d\alpha,
\end{equation*}
we therefore have
\begin{equation*}
R_-(P)\leq N^*(\tau) \leq R_+(P).
\end{equation*}
Consequently, it suffices to prove that
\begin{equation*}
R_\pm(P)=2\eta\Gamma(1+1/k)^s\Gamma(s/k)^{-1}P^{s-k}+o(P^{s-k}).
\end{equation*}

In the approximations which follow, we need to use the moduli of the above kernel functions, and as such it is helpful to note the following decomposition (see \cite[Section 2]{excepsets}). We write
\begin{equation} \label{Kdecomp}
\left|K_\pm(\alpha)\right|^2=K_1(\alpha)K_2^{\pm}(\alpha),
\end{equation} 
where
\begin{equation}\label{reallyK1}
K_1(\alpha)=\Bigg({\dfrac{\sin(\pi\alpha\delta)}{\pi\alpha\delta}}\Bigg)^2 = \delta^{-1}K(\alpha;\delta)
\end{equation} 
and
\begin{equation}\label{reallyK2}
K_2^{\pm}(\alpha)=\Bigg({\dfrac{\sin(\pi\alpha(2\eta\pm\delta))}{\pi\alpha}}\Bigg)^2 = (2\eta\pm\delta)K(\alpha;2\eta\pm\delta)
\end{equation} 
are both non-negative.

Using (\ref{Kbound}), we also note that
\begin{align} \label{K1est}
\int_\mathbb{R}K_1(\alpha)e(\alpha t)\,d\alpha&=\begin{cases}
\delta^{-1}(1-\delta^{-1}\left|t\right|), &\text{if } \left|t\right|<\delta,\\
0, &\text{otherwise,}\\
\end{cases}
\end{align}
and
\begin{align} \label{K2est}
\int_\mathbb{R}K_2^\pm(\alpha)e(\alpha t)\,d\alpha&=\begin{cases}
2\eta\pm\delta-\left|t\right|, &\text{if } \left|t\right|<2\eta\pm\delta,\\
0, &\text{otherwise.}\\
\end{cases}
\end{align}

\section{An auxiliary estimate} \label{Aminorarc}
In this section, we achieve a bound on the contribution from the traditional Hardy--Littlewood minor arcs, namely those points which are not close to a rational number with small denominator. In doing so, we improve on Chow's result for the number of variables required for the asymptotic formula (\ref{asymp}) to hold. We follow closely the method of Wooley in~\cite[Section 2]{asympform}. Thus, we firstly obtain an estimate for a related mean value, in the case where we have a single shift $\theta=\theta_1=\dotsc=\theta_s$. We then use this result, along with H\"older's inequality, to bound the quantity we are interested in, and to generalise to the case in which the shifts need not be the same.

It is convenient to introduce some further notation for use in this section. We define the exponential sums
\begin{equation*}
g(\bm{\alpha})=g_k(\bm{\alpha},\theta;P)=\sum_{1\leq x\leq P}e(\alpha_1 x+\dotsc +\alpha_{k-1}x^{k-1}+\alpha_k(x-\theta)^k),
\end{equation*}
and
\begin{equation*}
G(\bm{\beta},\mu)=G_k(\bm{\beta},\mu,\theta;P)=\sum_{1\leq x\leq P}e(\beta_1 x+\dotsc +\beta_{k-2}x^{k-2}+\mu(x-\theta)^k),
\end{equation*}
as well as the polynomials
\begin{equation*}
\sigma_{s,j}(\mathbf{x})=\sum_{i=1}^s(x_i^j-x_{s+i}^j), \quad (1\leq j\leq k-1),
\end{equation*}
and
\begin{equation*}
\sigma_{s,k}(\mathbf{x})=\sum_{i=1}^s({(x_i-\theta)}^k-{(x_{s+i}-\theta)}^k).
\end{equation*}
We use $\oint$ to denote the integral over $[0,1]^t$ for a suitable value of $t$, and we define the integral
\begin{equation*}
I^\pm_{s,k}(P,\theta)=\int_{\mathbb{R}} \left|f_\theta(\alpha)\right|^{2s}\left|K_\pm(\alpha)\right|\,d\alpha.
\end{equation*}
Let $J_{s,k}(P,\theta)$ be the number of solutions of the system
\begin{equation} \label{Jskshift}
\begin{cases}
\sigma_{s,j}(\mathbf{x})=0$,\quad $(1\leq j\leq k-1),\\
\left|\sigma_{s,k}(\mathbf{x})\right|<\eta,
\end{cases}
\end{equation}
with $1\leq\bm{x}\leq P$.
Using binomial expansions, and the fact that $\eta\leq 1$, we see that this system is equivalent to the system of Diophantine equations
\begin{equation} \label{Jsknormal}
\begin{cases}
\sigma_{s,j}(\mathbf{x})=0$,\quad $(1\leq j\leq k-1),\\
\sum_{i=1}^s(x_i^k-x_{s+i}^k)=0.
\end{cases}
\end{equation}

Letting $J_{s,k}(P)$ denote the number of solutions to (\ref{Jsknormal}) with $1\leq\bm{x}\leq P$, it is therefore the case that $J_{s,k}(P,\theta)=J_{s,k}(P)$. The quantity $J_{s,k}(P)$ has been widely studied, originally by Vinogradov in the 1930s, leading ultimately to the recent result of Bourgain, Demeter and Guth (see \cite[Theorem 1.1]{BDG}), who prove that
\begin{equation} \label{VMVT}
J_{s,k}(P)\ll P^{s+\epsilon}+P^{2s-\frac{1}{2}k(k+1)+\epsilon}
\end{equation}
for all $s\geq 1$ and $k\geq 4$. The case $k=3$ is due to Wooley in \cite{wooleyk3}.

For $1\leq Q\leq P$, we define $\mathfrak{v}=\mathfrak{v}_Q$ as in (\ref{vdef}). In later applications we will consider $Q=(2k)^{-1}P^{1/4}$. We are interested in an estimate for the minor arc portion (in the Hardy--Littlewood sense) of the integral $I^\pm_{s,k}(P,\theta)$. For $B\subset\mathbb{R}$, we write
\begin{equation} \label{Ipm}
I^\pm(B)=I^\pm_{s,k}(B,P,\theta)=\int_{B} \left|f_\theta(\alpha)\right|^{2s}\left|K_\pm(\alpha)\right|\,d\alpha.
\end{equation}
This allows us to state the key result of this section.
\begin{theorem} \label{minorarcs}
For natural numbers $s\geq 2$, $k\geq 2$, and for $\theta\in(0,1)$, we have \begin{equation*}
I^\pm(\mathfrak{v}) \ll  P^{\epsilon}Q^{-1}(P^{s+\frac{1}{2}k(k-1)}+P^{2s-k}).
\end{equation*}
\end{theorem}
\begin{proof}
We would like to rewrite the integral of interest in terms of the function $G(\bm{\beta},\mu)$, in order to separate out the $x^{k-1}$ term and estimate it using the rational approximation properties of points in $\mathfrak{v}$.

For $\mathbf{h}\in\mathbb{Z}^{k-2}$, let \begin{equation*}
\delta(\mathbf{x},\mathbf{h})=\prod_{j=1}^{k-2}\Big(\int_0^1 e(\beta_j(\sigma_{s,j}(\mathbf{x})-h_j))\,d\beta_j\Big),
\end{equation*}
which, by orthogonality, is equal to $1$ if $\sigma_{s,j}(\mathbf{x})=h_j$ for all $1\leq j\leq k-2$, and zero otherwise.

For any fixed $\mathbf{x}\in [1,P]^{2s}$, there is precisely one choice of $\mathbf{h}\in\mathbb{Z}^{k-2}$ which satisfies the above condition, and by the definition of $\sigma_{s,j}$ we have $\left|\sigma_{s,j}(\mathbf{x})\right|\leq sP^j$ for $1\leq j\leq k-2$. Hence \begin{equation*}
\sum_{\left|h_1\right|\leq sP}\dotsc\sum_{\left|h_{k-2}\right|\leq sP^{k-2}}\delta(\mathbf{x},\mathbf{h})=1.
\end{equation*}

We can therefore rewrite the minor arc integral in the form \begin{align*}
I^\pm(\mathfrak{v})&= \int_{\mathfrak{v}}\sum_{1\leq \mathbf{x}\leq P}e(\mu\sigma_{s,k}(\mathbf{x}))\left|K_\pm(\mu)\right|\,d\mu\\
&=\sum_{\bm{h}}\sum_{1\leq \mathbf{x}\leq P}\delta(\mathbf{x},\mathbf{h})\int_{\mathfrak{v}}e(\mu\sigma_{s,k}(\mathbf{x}))\left|K_\pm(\mu)\right|\,d\mu,
\end{align*}
where the first summation is over $(k-2)$-tuples $\bm{h}$ satisfying $\left|h_i\right|\leq sP^i$ for $1\leq i\leq k-2$. From the definition of $G(\bm{\beta},\mu)$, we obtain
\begin{align*}
I^\pm(\mathfrak{v})&=\sum_{\bm{h}}\sum_{1\leq \mathbf{x}\leq P}\prod_{j=1}^{k-2}\Big(\int_0^1 e(\beta_j(\sigma_{s,j}(\mathbf{x})-h_j))\,d\beta_j\Big)\int_{\mathfrak{v}}e(\mu\sigma_{s,k}(\mathbf{x}))\left|K_\pm(\mu)\right|\,d\mu\\
&=\sum_{\bm{h}}\int_{\mathfrak{v}}\oint\left|G(\bm{\beta},\mu)\right|^{2s}e(-\bm{\beta}\cdot\bm{h})\left|K_\pm(\mu)\right|\,d\bm{\beta}\,d\mu.
\end{align*}
Hence, using the triangle inequality, and defining
\begin{equation} \label{Iindepy}
\mathcal{I}=\int_{\mathfrak{v}}\oint \left|G(\bm{\beta},\mu)\right|^{2s}\left|K_\pm(\mu)\right|\,d\bm{\beta}\,d\mu,
\end{equation}
we see that
\begin{align} \label{eq: 1stsplit}
I^\pm(\mathfrak{v}) &\leq\sum_{\bm{h}}\int_{\mathfrak{v}}\oint\left|G(\bm{\beta},\mu)\right|^{2s}\left|K_\pm(\mu)\right|\,d\bm{\beta}\,d\mu\nonumber\\
&\ll P^{\frac{1}{2}(k-1)(k-2)}\mathcal{I}.
\end{align}

Similarly, writing
\begin{equation*}
g(\bm{\alpha},\mu)=\sum_{1\leq x\leq P} e(\alpha_1x+\dotsc+\alpha_{k-1}x^{k-1}+\mu(x-\theta)^k),
\end{equation*}
we have 
\begin{equation} \label{Fint}
\mathcal{I}=\sum_{\left|h\right|\leq sP^{k-1}}\int_{\mathfrak{v}}\oint\left|g(\bm{\alpha},\mu)\right|^{2s}e(-\alpha_{k-1}h)\left|K_\pm(\mu)\right|\,d\bm{\alpha}\,d\mu.
\end{equation}

Let $\psi(z;\bm{\alpha})=\alpha_1z+\dotsc+\alpha_{k-1}z^{k-1}+\alpha_k(z-\theta)^k$, so that, with a shift of variables, we have\begin{equation} \label{eq: psishift}
g(\bm{\alpha})=\sum_{1\leq x\leq P}e(\psi(x;\bm{\alpha}))=\sum_{1+y\leq x\leq P+y}e(\psi(x-y;\bm{\alpha})).
\end{equation}

Let \begin{equation*}
\mathcal{L}(\gamma)=\sum_{1\leq z\leq P}e(-\gamma z),
\end{equation*}
and
\begin{equation*}
\mathfrak{g}_y(\bm{\alpha};\gamma)=\sum_{1\leq x\leq 2P} e(\psi(x-y;\bm{\alpha})+\gamma(x-y)),
\end{equation*}
so that
\begin{equation*}
\overline{\mathfrak{g}}_y(\bm{\alpha};\gamma)=\mathfrak{g}_y(\bm{-\alpha};-\gamma).
\end{equation*}

Then, for $1\leq y\leq P$, we observe from (\ref{eq: psishift}) that
\begin{align*}
\int_0^1 \mathfrak{g}_y(\bm{\alpha};\gamma)\mathcal{L}(\gamma)\,d\gamma &=\int_0^1 \sum_{1\leq x\leq 2P}\sum_{1\leq z\leq P} e(\psi(x-y;\bm{\alpha})+\gamma(x-y-z))\,d\gamma\\
&=\sum_{1\leq x\leq 2P}\sum_{\substack{1\leq z\leq P\\
					z=x-y}
}e(\psi(x-y;\bm{\alpha}))\\
&=g(\bm{\alpha}).
\end{align*}

Substituting this relation into (\ref{Fint}), we find that
\begin{equation*}
\mathcal{I}= \sum_{\left|h\right|\leq sP^{k-1}} \int_{\mathfrak{v}}\oint\left|\int_0^1 \mathfrak{g}_y(\bm{\alpha},\mu;\gamma)\mathcal{L}(\gamma)\,d\gamma\right|^{2s}e(-\alpha_{k-1}h)\left|K_\pm(\mu)\right|\,d\bm{\alpha}\,d\mu.
\end{equation*}
Writing
\begin{equation*}
\mathcal{G}_y(\bm{\alpha},\mu;\bm{\gamma})=\prod_{i=1}^s \mathfrak{g}_y(\bm{\alpha},\mu;\gamma_i)\overline{\mathfrak{g}}_y(\bm{\alpha},\mu;\gamma_{s+i}),
\end{equation*}
and
\begin{equation*}
\tilde{\mathcal{L}}(\bm{\gamma})=\prod_{i=1}^s\mathcal{L}(\gamma_i)\mathcal{L}(-\gamma_{s+i}),
\end{equation*}
we see that
\begin{equation*}
\mathcal{I}=\sum_{\left|h\right|\leq sP^{k-1}} \int_{\mathfrak{v}}\oint\oint\mathcal{G}_y(\bm{\alpha},\mu;\bm{\gamma})\tilde{\mathcal{L}}(\bm{\gamma})e(-\alpha_{k-1}h)\left|K_\pm(\mu)\right|\,d\bm{\gamma}\,d\bm{\alpha}\,d\mu.
\end{equation*}

If we let \begin{equation} \label{eq: defIh}
I_h(\bm{\gamma},y)=\int_{\mathfrak{v}}\oint\mathcal{G}_y(\bm{\alpha},\mu;\bm{\gamma})e(-\alpha_{k-1}h)\left|K_\pm(\mu)\right|\,d\bm{\alpha}\,d\mu,
\end{equation}
then we can write \begin{equation} \label{eq: eighteen}
\mathcal{I}= \sum_{\left|h\right|\leq sP^{k-1}} \oint I_h(\bm{\gamma},y)\tilde{\mathcal{L}}(\bm{\gamma})\,d\bm{\gamma}.
\end{equation}

Evaluating the inner integral of $I_h(\bm{\gamma},y)$ using orthogonality, we see that
\begin{equation} \label{eq: deltasum}
\oint\mathcal{G}_y(\bm{\alpha},\mu;\bm{\gamma})e(-\alpha_{k-1}h)\left|K_\pm(\mu)\right|\,d\bm{\alpha} = \left|K_\pm(\mu)\right|\sum_{1\leq \mathbf{x}\leq 2P}\Delta_{\bm{x}}(\mu,\bm{\gamma},h,y),
\end{equation}
where \begin{equation*}
\Delta_{\bm{x}}(\mu,\bm{\gamma},h,y)= e(\mu\sigma_{s,k}(\mathbf{x}-y)+\sum_{i=1}^s (\gamma_i(x_i-y)-\gamma_{s+i}(x_{s+i}-y)))
\end{equation*}
whenever \begin{equation} \label{eq: deltacond}
\begin{cases}
\sigma_{s,j}(\mathbf{x}-y)=0,\;\;\; (1\leq j\leq k-2),\\
\sigma_{s,k-1}(\mathbf{x}-y)=h,
\end{cases}
\end{equation}
and otherwise $\Delta_{\bm{x}}(\mu,\bm{\gamma},h,y)=0$.

Using binomial expansions, we see that whenever the above conditions (\ref{eq: deltacond}) hold, we also have the relations
\begin{equation*}
\sigma_{s,j}(\mathbf{x})=0=\sigma_{s,j}(\mathbf{x}-\theta),
\end{equation*}
for $1\leq j\leq k-2$, and
\begin{equation*}
\sigma_{s,k-1}(\mathbf{x})=h=\sigma_{s,k-1}(\mathbf{x}-\theta),
\end{equation*}
and consequently
\begin{equation*}
\sigma_{s,k}(\mathbf{x}-y)= \sum_{i=1}^s ((x_i-\theta-y)^k-(x_{s+i}-\theta-y)^k) = \sigma_{s,k}(\mathbf{x})-khy.
\end{equation*}

We therefore see from (\ref{eq: deltasum}) that \begin{align*}
\oint\mathcal{G}_y(\bm{\alpha},\mu;\bm{\gamma})&e(-\alpha_{k-1}h)\left|K_\pm(\mu)\right|\,d\bm{\alpha} \\
&\leq \oint \left|K_\pm(\mu)\right|\mathcal{G}_0(\bm{\alpha},\mu;\bm{\gamma})e(-\mu khy-\alpha_{k-1}h)\,\omega_{y,\bm{\gamma}}\,d\bm{\alpha},
\end{align*}
where $\omega_{y,\bm{\gamma}} = e(-y\sigma_{s,1}(\bm{\gamma}))$.
Using (\ref{eq: defIh}), we have
\begin{align*}
\sum_{\left|h\right|\leq sP^{k-1}}&I_h(\bm{\gamma},y)\\
&\leq \int_{\mathfrak{v}}\oint \left|K_\pm(\mu)\right|\mathcal{G}_0(\bm{\alpha},\mu;\bm{\gamma})\sum_{\left|h\right|\leq sP^{k-1}}e(-\mu khy-\alpha_{k-1}h)\,\omega_{y,\bm{\gamma}}\,d\bm{\alpha}\,d\mu\\
&\ll \int_{\mathfrak{v}}\oint \left|K_\pm(\mu)\right|\left|\mathcal{G}_0(\bm{\alpha},\mu;\bm{\gamma})\right| \min\{P^{k-1}, \|\mu ky+\alpha_{k-1}\|^{-1}\}\,\,d\bm{\alpha}\,d\mu
\end{align*}
by a standard estimate for exponential sums (see, for example, \cite[Chapter 3]{davenport}).

Averaging over all permitted values of $y$, and writing
\begin{equation*}
\Psi(\mu,\alpha_{k-1})=P^{-1}\sum_{1\leq y\leq P}\min\{P^{k-1}, \|\mu ky+\alpha_{k-1}\|^{-1}\},
\end{equation*}
we see that \begin{align} \label{eq: average}
P^{-1}\sum_{1\leq y\leq P}\sum_{\left|h\right|\leq sP^{k-1}}& I_h(\bm{\gamma},y)\nonumber\\
&\ll \int_{\mathfrak{v}}\oint \left|K_\pm(\mu)\right|\left|\mathcal{G}_0(\bm{\alpha},\mu;\bm{\gamma})\right| \Psi(\mu,\alpha_{k-1})\,\,d\bm{\alpha}\,d\mu. 
\end{align}

Now we find a rational approximation for $\mu$. By Dirichlet's approximation theorem, there exist $b\in\mathbb{Z}$ and $r\in\mathbb{N}$ with $(b,r)=1$ such that $r\leq P^k Q^{-1}$ and $\left|r\mu-b\right|\leq QP^{-k}\leq r^{-1}$. Using a modification of \cite[Lemma 3.2]{baker}, we have
\begin{equation*}
\Psi(\mu,\alpha_{k-1}) \ll P^{k-1}(P^{-1}+r^{-1}+rP^{-k})\log(2r).
\end{equation*}
By the definition of $\mathfrak{v}$, we have $r>Q$, and therefore \begin{equation*}
\sup_{\mu\in\mathfrak{v}} \Psi(\mu,\alpha_{k-1}) \ll Q^{-1}P^{k-1}\log{P}.
\end{equation*}

Substituting this into (\ref{eq: average}) and using H\"{o}lder's inequality, we see that \begin{align} \label{Ups}
\quad\quad P^{-1}&\sum_{1\leq y\leq P}\sum_{\left|h\right|\leq sP^{k-1}} I_h(\bm{\gamma},y)\nonumber\\ 
&\ll Q^{-1}P^{k-1}(\log{P})\int_{\mathfrak{v}}\oint \left|K_\pm(\mu)\right|\left|\prod_{i=1}^s \mathfrak{g}_0(\bm{\alpha},\mu;\gamma_i)\overline{\mathfrak{g}}_0(\bm{\alpha},\mu;\gamma_{s+i})\right| \,d\bm{\alpha}\,d\mu\nonumber\\
&\ll Q^{-1}P^{k-1}(\log{P})\prod_{i=1}^{2s}\Bigg(\int_{\mathfrak{v}}\oint \left|K_\pm(\mu)\right|\left|\mathfrak{g}_0(\bm{\alpha},\mu;\gamma_i)\right|^{2s} \,d\bm{\alpha}\,d\mu\Bigg)^{1/2s}\nonumber\\
&\ll Q^{-1}P^{k-1}(\log{P})\sup_{\gamma\in[0,1)}\int_{\mathbb{R}}\oint \left|K_\pm(\mu)\right|\left|\mathfrak{g}_0(\bm{\alpha},\mu;\gamma)\right|^{2s} \,d\bm{\alpha}\,d\mu\nonumber\\
&\ll Q^{-1}P^{k-1}(\log{P})\int_{\mathbb{R}}\oint \left|K_\pm(\mu)\right|\left|g_k(\bm{\alpha},\mu,\theta;2P)\right|^{2s} \,d\bm{\alpha}\,d\mu.
\end{align}

For a general function $H\colon\mathbb{R}\to\mathbb{R}$, we write
\begin{equation*}
\Upsilon(H)=\int_{\mathbb{R}}\left|H(\mu)\right|\left|g_k(\bm{\alpha},\mu,\theta;2P)\right|^{2s} \,d\mu.
\end{equation*}
Using the Cauchy--Schwarz inequality, and the decomposition (\ref{Kdecomp}), we obtain
\begin{equation*}
\oint\Upsilon(K_\pm)\,d\bm{\alpha} \leq \bigg(\oint\Upsilon(K_1)\,d\bm{\alpha}\bigg)^{1/2} \bigg(\oint\Upsilon(K_2^\pm)\,d\bm{\alpha}\bigg) ^{1/2}.
\end{equation*}
From (\ref{K1est}), we deduce that $\Upsilon(K_1)$ contributes
\begin{equation*}
\delta^{-1}(1-\delta^{-1}\left|\sigma_{s,k}(\bm{x})\right|)\,e(\alpha_1\sigma_{s,1}(\bm{x})+\dotsc+\alpha_{k-1}\sigma_{s,k-1}(\bm{x}))
\end{equation*}
whenever $\left|\sigma_{s,k}(\bm{x})\right|<\delta$. Recalling that $\delta=\eta L(P)^{-1}\leq\eta$ for sufficiently large $P$, and using the equivalence of systems (\ref{Jskshift}) and (\ref{Jsknormal}), this implies that
\begin{equation*}
\oint\Upsilon(K_1)\,d\bm{\alpha} \leq \delta^{-1}J_{s,k}(2P) \ll L(P)J_{s,k}(2P).
\end{equation*}
Similarly, using (\ref{K2est}), we have
\begin{equation*}
\oint\Upsilon(K_2^\pm)\,d\bm{\alpha} \ll J_{s,k}(2P).
\end{equation*}
We remark that we also have $\oint\Upsilon(K)\,d\bm{\alpha} \ll J_{s,k}(2P)$, which allows us to establish the simplified claim (\ref{msimp}) given in the introduction to this paper.

Substituting the above estimates into (\ref{Ups}) and using (\ref{LPdefn}), we see that
\begin{equation*}
P^{-1}\sum_{1\leq y\leq P}\sum_{\left|h\right|\leq sP^{k-1}} I_h(\bm{\gamma},y) \ll Q^{-1}P^{k-1}(\log{P})^{3/2}\,J_{s,k}(2P).
\end{equation*}
Returning to (\ref{eq: eighteen}), and noting that $\mathcal{I}$ as originally defined in (\ref{Iindepy}) does not depend on $y$, we see that
\begin{equation}\label{eq: nearly}
\mathcal{I} = P^{-1}\sum_{1\leq y\leq P}\mathcal{I}\ll Q^{-1}P^{k-1}(\log{P})^{3/2}\,J_{s,k}(2P) \oint\left|\tilde{\mathcal{L}}(\bm{\gamma})\right|\,d\bm{\gamma}. 
\end{equation}
By the definition of $\mathcal{L}(\gamma)$, we have
\begin{equation*}
\int_0^1 \left|\mathcal{L}(\gamma)\right|\,d\gamma \leq \int_0^1 \min\{P,\|\gamma\|^{-1}\}\,d\gamma \ll \log{P},
\end{equation*}
and therefore
\begin{equation*}
\oint\left|\tilde{\mathcal{L}}(\bm{\gamma})\right|\,d\bm{\gamma}=\oint\left|\prod_{i=1}^s\mathcal{L}(\gamma_i)\mathcal{L}(-\gamma_{s+i})\right|\,d\bm{\gamma} \ll (\log{P})^{2s}.
\end{equation*}
Substituting this into (\ref{eq: nearly}), we see that
\begin{equation*}
\mathcal{I} \ll Q^{-1}P^{k-1}(\log{P})^{2s+3/2}\,J_{s,k}(2P),
\end{equation*}
and hence, from (\ref{eq: 1stsplit}), that
\begin{align*}
I^\pm(\mathfrak{v})&\ll Q^{-1}P^{\frac{1}{2}k(k-1)}(\log{P})^{2s+3/2}\,J_{s,k}(2P)\\
&\ll Q^{-1}P^{\frac{1}{2}k(k-1)+\epsilon}\,J_{s,k}(2P).
\end{align*}

Using (\ref{VMVT}), we conclude that
\begin{equation*}
I^\pm(\mathfrak{v}) \ll  P^{\epsilon}Q^{-1}(P^{s+\frac{1}{2}k(k-1)}+P^{2s-k}),
\end{equation*}\\
as required.
\end{proof}

In particular, we have 
\begin{equation*}
\int_{\mathfrak{v}} \left|f_\theta(\mu)\right|^{2s}\left|K_\pm(\mu)\right|\,d\mu \ll  Q^{-1}P^{s+\frac{1}{2}k(k-1)+\epsilon}
\end{equation*}
whenever $s\leq\tfrac{1}{2}k(k+1)$, and

\begin{equation*}
\int_{\mathfrak{v}} \left|f_\theta(\mu)\right|^{2s}\left|K_\pm(\mu)\right|\,d\mu \ll Q^{-1} P^{2s-k+\epsilon}
\end{equation*}
whenever $s\geq \tfrac{1}{2}k(k+1)$.

We now wish to use the above result to bound the minor arc contribution for our shifted Waring's problem. From this point onwards, we fix $Q=(2k)^{-1}P^{1/4}$.

In Corollary \ref{altcorollary}, we use the Cauchy--Schwarz inequality and a trivial estimate in order to limit the number of variables needed to achieve the required bound, which ultimately allows us to prove Theorem \ref{mainThm} (in Section \ref{majorarc}). We then go on to provide a conjectural further improvement (for $k=10$ and $k\geq 12$) based on an adaptation of a theorem of Bourgain (arising from the results in \cite{BDG}).

\begin{corollary} \label{altcorollary}
Let $k\geq 2$ be a natural number, and let $s_0(k)=k^2+(3k-1)/4$. Then for any natural number $s\geq s_0(k)$, we have
\begin{equation} \label{minorestALT}
\int_{\mathfrak{v}}\left|f_\theta(\alpha)\right|^s\left|K_\pm(\alpha)\right|\,d\alpha =o(P^{s-k}).
\end{equation}
\end{corollary}
\begin{proof}
Fix $k\geq 2$, and let $s_0=s_0(k)$. We first prove (\ref{minorestALT}) in the case $s=s_0$.
Let \begin{equation*}
a=\frac{s_0-2}{(k+2)(k-1)}, \quad b=\frac{k^2+k-s_0}{(k+2)(k-1)}.
\end{equation*}
Note that by the definition of $s_0$, and since $k>5/3$, we have
\begin{equation} \label{4k-3}
b= \frac{k+1}{4k^2+4k-8} <\frac{k+1}{4k^2+k-3}=\frac{1}{4k-3}.
\end{equation}
We have $a+b=1$, and $ak(k+1)+2b=s_0$, so, using the notation introduced in (\ref{Ipm}), and suppressing the dependence on $k, P$ and $\theta$, we can apply H\"older's inequality to see that
\begin{equation*}
\int_{\mathfrak{v}}\left|f_\theta(\alpha)\right|^{s_0}\left|K_\pm(\alpha)\right|\,d\alpha \ll \big(I^\pm_{k(k+1)/2}(\mathfrak{v})\big)^a \big(I^\pm_1(\mathfrak{v})\big)^b.
\end{equation*}

We evaluate the first term using Theorem \ref{minorarcs} to get
\begin{equation*}
I^\pm_{k(k+1)/2}(\mathfrak{v}) \ll Q^{-1}P^{k(k+1)-k+\epsilon}.
\end{equation*}
For the second term, we use the decomposition (\ref{Kdecomp}), along with the Cauchy--Schwarz inequality, to obtain
\begin{equation*}
I^\pm_1(\mathfrak{v}) \ll \bigg(\int_{\mathfrak{v}}\left|f_\theta(\alpha)\right|^2 K_1(\alpha)\,d\alpha\bigg)^{1/2} \bigg(\int_{\mathfrak{v}}\left|f_\theta(\alpha)\right|^2 K^\pm_2(\alpha)\,d\alpha\bigg)^{1/2}.
\end{equation*}

Since the number of solutions to the inequality $\left|(x-\theta)^k-(y-\theta)^k\right|<\delta$ with $1\leq x,y\leq P$ is $O(P)$, we use (\ref{K1est}) and (\ref{LPdefn}) to see that
\begin{equation*}
\int_{\mathfrak{v}}\left|f_\theta(\alpha)\right|^2 K_1(\alpha)\,d\alpha \leq \int_{\mathbb{R}}\left|f_\theta(\alpha)\right|^2 K_1(\alpha)\,d\alpha \ll L(P)P.
\end{equation*}
Similarly, using (\ref{K2est}), we have
\begin{equation*}
\int_{\mathfrak{v}}\left|f_\theta(\alpha)\right|^2 K^\pm_2(\alpha)\,d\alpha \leq \int_{\mathbb{R}}\left|f_\theta(\alpha)\right|^2 K^\pm_2(\alpha)\,d\alpha \ll P.
\end{equation*}
We therefore see that
\begin{equation*}
I^\pm_1(\mathfrak{v})  \ll (\log{P})^{1/2}P \ll P^{1+\epsilon}.
\end{equation*}

Hence, with some rearrangement, and using the definitions of $a$, $b$ and $Q$,
\begin{align*}
\int_{\mathfrak{v}}\left|f_\theta(\alpha)\right|^{s_0}\left|K_\pm(\alpha)\right|\,d\alpha &\ll P^{a(k(k+1)-k-1/4+\epsilon)+b(1+\epsilon)}\\
&=P^{s_0-k+\epsilon-\iota},
\end{align*}
where 
\begin{equation*}
\iota=a/4-b(k-1)=1/4-b(k-3/4)>\epsilon
\end{equation*}
for small enough $\epsilon$, by (\ref{4k-3}).

For $s>s_0$, we then use the trivial estimate to obtain
\begin{align*}
\int_{\mathfrak{v}}\left|f_\theta(\alpha)\right|^s\left|K_\pm(\alpha)\right|\,d\alpha &\ll P^{s-s_0}\int_{\mathfrak{v}}\left|f_\theta(\alpha)\right|^{s_0}\left|K_\pm(\alpha)\right|\,d\alpha\\
&\ll P^{s-s_0}P^{s_0-k+\epsilon-\iota}=o(P^{s-k}).
\end{align*}
\end{proof}

We now present a more sophisticated version of the above argument, which follows a similar structure. For $j<k$ a natural number, we define
\begin{align*}
s_1(k,j)&=\Bigg\lceil k(k+1)-\frac{k(k+1)-j(j+1)}{4(k-j)+1} \Bigg\rceil +1\\
&=k^2+k+1-\Bigg\lfloor\frac{k(k+1)-j(j+1)}{4(k-j)+1} \Bigg\rfloor.
\end{align*}

We require an improved version of Hua's lemma. Since \cite[Theorem 4.1]{BDG} applies equally to the case of exponential sums of suitably separated points, such as the set $\{x-\theta:x\in\mathbb{N}\}$, as it does to the integer case, it would seem that the following `shifted' analogue of \cite[Theorem 10]{bourgainexpl} should hold. However, the details of such a result do not yet appear in the literature.

\begin{hypothesis}[``Shifted Hua's Lemma'']
For $j\leq k$ a natural number, and for any fixed, positive $\zeta$, we have
\begin{equation} \label{conj}
\int_{\mathbb{R}}\left|f_\theta(\alpha)\right|^{j(j+1)}K(\alpha;\zeta)\,d\alpha \ll P^{j^2+\epsilon}.
\end{equation}
\end{hypothesis}
Note that the implicit constant in (\ref{conj}) may depend on $\zeta$.

\begin{corollary} \label{corollary}
Assuming the shifted Hua's lemma, for any natural number $s\geq s_1(k,j)$ we have
\begin{equation} \label{minorest}
\int_{\mathfrak{v}}\left|f_\theta(\alpha)\right|^s\left|K_\pm(\alpha)\right|\,d\alpha =o(P^{s-k}).
\end{equation}
\end{corollary}
\begin{proof}
Fix $j$ and $k$, and let $s_1=s_1(k,j)$. We first prove (\ref{minorest}) in the case $s=s_1$.
Let \begin{equation*}
a=\frac{s_1-j(j+1)}{k(k+1)-j(j+1)}, \,\,b=\frac{k(k+1)-s_1}{k(k+1)-j(j+1)}.
\end{equation*}
Note that by the definition of $s_1$ we have
\begin{equation} \label{bsmall}
b= \frac{k(k+1)-\Big\lceil k(k+1)-\frac{k(k+1)-j(j+1)}{4(k-j)+1} \Big\rceil-1}{k(k+1)-j(j+1)} <\frac{1}{4(k-j)+1}.
\end{equation}
We have $a+b=1$, and $ak(k+1)+bj(j+1)=s_1$, so, as in Corollary \ref{altcorollary}, we can apply H\"older's inequality to see that
\begin{equation*}
\int_{\mathfrak{v}}\left|f_\theta(\alpha)\right|^{s_1}\left|K_\pm(\alpha)\right|\,d\alpha \ll
\big(I^\pm_{k(k+1)/2}(\mathfrak{v})\big)^a \big(I^\pm_{j(j+1)/2}(\mathfrak{v})\big)^b.
\end{equation*}

We evaluate the first term using Theorem \ref{minorarcs} to get
\begin{equation*}
I^\pm_{k(k+1)/2}(\mathfrak{v}) \ll Q^{-1}P^{k(k+1)-k+\epsilon}.
\end{equation*}
For the second term, as in Corollary \ref{altcorollary}, we obtain
\begin{align*}
I^\pm_{j(j+1)/2}(\mathfrak{v}) &\ll \bigg(\int_{\mathfrak{v}}\left|f_\theta(\alpha)\right|^{j(j+1)} K_1(\alpha)\,d\alpha\bigg)^{1/2} \bigg(\int_{\mathfrak{v}}\left|f_\theta(\alpha)\right|^{j(j+1)} K^\pm_2(\alpha)\,d\alpha\bigg)^{1/2} \\
&\ll \bigg(\int_{\mathbb{R}}\left|f_\theta(\alpha)\right|^{j(j+1)} K_1(\alpha)\,d\alpha\bigg)^{1/2} \bigg(\int_{\mathbb{R}}\left|f_\theta(\alpha)\right|^{j(j+1)} K^\pm_2(\alpha)\,d\alpha\bigg)^{1/2}.
\end{align*}

Combining (\ref{reallyK1}) and (\ref{reallyK2}) with the shifted Hua's lemma, we see that
\begin{equation*}
\int_{\mathbb{R}}\left|f_\theta(\alpha)\right|^{j(j+1)} K_1(\alpha)\,d\alpha \ll L(P)P^{j^2+\epsilon} \ll P^{j^2+\epsilon},
\end{equation*}
and
\begin{equation*}
\int_{\mathbb{R}}\left|f_\theta(\alpha)\right|^{j(j+1)} K^\pm_2(\alpha)\,d\alpha \ll P^{j^2+\epsilon},
\end{equation*}
and therefore that
\begin{equation*}
I^\pm_{j(j+1)/2}(\mathfrak{v})\ll P^{j^2+\epsilon}.
\end{equation*}

Hence, with some rearrangement, and using the definitions of $a$, $b$ and $Q$,
\begin{align*}
\int_{\mathfrak{v}}\left|f_\theta(\alpha)\right|^{s_1}\left|K_\pm(\alpha)\right|\,d\alpha &\ll P^{(k(k+1)-k+\epsilon-1/4)a+(j(j+1)-j+\epsilon)b}\\
&=P^{s_1-k+\epsilon-\iota}
\end{align*}
where $\iota=1/4-(k-j+1/4)b>\epsilon$ for small enough $\epsilon$, by (\ref{bsmall}).

For $s>s_1$, we then use the trivial estimate to obtain
\begin{align*}
\int_{\mathfrak{v}}\left|f_\theta(\alpha)\right|^s\left|K_\pm(\alpha)\right|\,d\alpha &\ll P^{s-s_1}\int_{\mathfrak{v}}\left|f_\theta(\alpha)\right|^{s_1}\left|K_\pm(\alpha)\right|\,d\alpha\\
&\ll P^{s-s_1}P^{s_1-k+\epsilon-\iota}\\
&=o(P^{s-k}).
\end{align*}
\end{proof}

Differentiation shows that, for a given $k$, the minimal value of $s_1(k,j)$ occurs when
\begin{equation*}
j=j_0(k)=\Bigg[ k+\frac{1}{4}-\sqrt{\frac{1}{2}k+\frac{5}{16}}\Bigg]
\end{equation*}
where $[x]$ denotes the nearest integer to $x$. Note that for all $k\geq 2$, we have $j_0(k)<k$. Letting $s_1(k)=s_1(k,j_0(k))$, we therefore conclude the following corollary, noting that $s_1(k)=k^2+k/2+O(k^{1/2})$, and that $s_1(k)<s_0(k)$ for $k=10$ and $k\geq12$.
\begin{corollary} \label{thisone}
Assuming the shifted Hua's lemma, for any natural number $s\geq s_1(k)$, we have
\begin{equation*}
\int_{\mathfrak{v}}\left|f_\theta(\alpha)\right|^s\left|K_\pm(\alpha)\right|\,d\alpha =o(P^{s-k}).
\end{equation*}
\end{corollary}

Finally, we generalise the above results to the case of mixed shifts $\theta_1,\dotsc,\theta_s$.
\begin{corollary} \label{minormixed}
Suppose that $\theta_1,\dotsc,\theta_s\in(0,1)$, and write $\bm{\theta}=(\theta_1,\dotsc,\theta_s)$. Then for any natural number $s\geq s_0(k)$, we have
\begin{equation*}
\int_{\mathfrak{v}}\left|f_{\bm{\theta}}(\alpha)\right| \left|K_\pm(\alpha)\right|\,d\alpha = o(P^{s-k}).
\end{equation*}
Assuming the shifted Hua's lemma, the same result holds whenever $s\geq s_1(k)$.
\end{corollary}
\begin{proof}
By H\"older's inequality, and using Corollary \ref{altcorollary} or Corollary \ref{thisone}, as appropriate, we have
\begin{align*}
\int_{\mathfrak{v}}\left|f_{\bm{\theta}}(\alpha)\right| \left|K_\pm(\alpha)\right|\,d\alpha &\ll \prod_{i=1}^s\Bigg(\int_{\mathfrak{v}}\left|f_{\theta_i}(\alpha)\right|^s \left|K_\pm(\alpha)\right|\,d\alpha\Bigg)^{1/s}\\
&=o(P^{s-k}).
\end{align*}
\end{proof}

\section{The minor and trivial arcs} \label{restofminor}
On the minor and trivial arcs, we first demonstrate the estimate
\begin{equation*}
\int_{\mathfrak{m}\cup\mathfrak{t}}\left|f_{\theta_1}(\alpha)f_{\theta_2}(\alpha)f_{\theta_3}(\alpha)^{s-2} K_\pm(\alpha)\right|\,d\alpha=o(P^{s-k}),
\end{equation*}
for shifts $\theta_1,\theta_2,\theta_3\in (0,1)$ with $\theta_1$ irrational, and later use H\"older's inequality to obtain the general case. We subdivide our arcs into those points with good rational approximations and those without, in a manner reminiscent of the classical Hardy--Littlewood method, by defining
\begin{gather*}
\mathfrak{N}_{a,q}=\{\alpha\in\mathfrak{m}\cup\mathfrak{t}:\left|q\alpha-a\right|\leq QP^{-k}\}, \\
\mathfrak{N}=\bigcup_{\substack{1\leq q\leq Q,\\
					(a,q)=1}
}\mathfrak{N}_{a,q},\quad\mbox{and \quad}\mathfrak{n}=(\mathfrak{m}\cup\mathfrak{t})\setminus\mathfrak{N}.
\end{gather*}

Those points in $\mathfrak{n}$ are handled using Corollary \ref{minormixed}, since $\mathfrak{n}=\mathfrak{v}\cap(\mathfrak{m}\cup\mathfrak{t})$, while those in $\mathfrak{N}$ are subdivided yet again on the basis of the size of the exponential sum $f_{\theta_3}(\alpha)$. For some real number $t$ (to be chosen later) satisfying $2k(k-1)t<1$, let
\begin{equation*}
\mathfrak{B}=\mathfrak{B}(t)=\{\alpha\in\mathfrak{N}:\left|f_{\theta_3}(\alpha)\right|\geq P^{1-t}\},
\end{equation*}
and
\begin{equation*}
\overline{\mathfrak{B}}=\overline{\mathfrak{B}}(t)=\mathfrak{N}\setminus\mathfrak{B},
\end{equation*}
so that \begin{equation*}
\mathfrak{m}\cup\mathfrak{t}=\mathfrak{B}\cup\overline{\mathfrak{B}}\cup\mathfrak{n}.
\end{equation*}

Let $\mathfrak{B}_v$, $\overline{\mathfrak{B}}_v$ denote  the intersection of $\mathfrak{B}$, $\overline{\mathfrak{B}}$ respectively with the unit interval $[v,1+v)$. Note that for any $v\in\mathbb{R}$, we have
\begin{align} \label{mesB}
\text{mes}(\mathfrak{B}_v\cup\overline{\mathfrak{B}}_v)\leq \sum_{q=1}^Q\sum_{a=1}^q 2QP^{-k}/q \ll Q^2 P^{-k}.
\end{align}
We use this to bound the contribution to the overall integral from $\overline{\mathfrak{B}}$.

\begin{lemma} \label{ulemma}
Let $t$ be such that $2k(k-1)t<1$. For $u>1/(2t)$, and for any $v\in\mathbb{R}$, we have
\begin{align*}
\int_{\overline{\mathfrak{B}}_v}\left|f_{\theta_3}(\alpha)\right|^u\,d\alpha =o(P^{u-k}).
\end{align*}
\end{lemma}
\begin{proof}
Note that, by assumption, we have $ut>1/2$. Therefore, using (\ref{mesB}), and recalling that $Q=(2k)^{-1}P^{1/4}$,
\begin{align*}
\int_{\overline{\mathfrak{B}}_v}\left|f_{\theta_3}(\alpha)\right|^u\,d\alpha &\ll (P^{1-t})^u \text{mes}(\overline{\mathfrak{B}}_v)\ll P^{u-ut}Q^2P^{-k}\\
&\ll P^{u-k+1/2-ut}=o(P^{u-k}).
\end{align*}
\end{proof}
We therefore have the following estimate for those $\overline{\mathfrak{B}}_v$ contained in the minor arcs.
\begin{lemma} \label{slemma}
For $s\geq k^2+2$, and for any $v$ with $\overline{\mathfrak{B}}_v \subset \mathfrak{m}$, there exists $\iota>0$ such that
\begin{align*}
\int_{\overline{\mathfrak{B}}_v}\left|f_{\theta_1}(\alpha)f_{\theta_2}(\alpha)f_{\theta_3}(\alpha)^{s-2}\right|\,d\alpha &\ll P^{s-k-\iota}\,T(P)^{-1}.
\end{align*}
\end{lemma}
\begin{proof}
By choosing $t$ so that $2k(k-1)t$ is as close as we like to 1, note that we can always find a $u$ such that $1/(2t)<u<k^2\leq s-2$. Applying Lemma \ref{ulemma} and (\ref{TPbound}), we see that
\begin{align*}
\int_{\overline{\mathfrak{B}}_v}\left|f_{\theta_1}(\alpha)f_{\theta_2}(\alpha)f_{\theta_3}(\alpha)^{s-2}\right|\,d\alpha &\ll \sup_{\alpha\in\overline{\mathfrak{B}}_v}\left|f_{\theta_1}(\alpha)f_{\theta_2}(\alpha)f_{\theta_3}(\alpha)^{s-2-u}\right|  \int_{\overline{\mathfrak{B}}_v}\left|f_{\theta_3}(\alpha)\right|^u\,d\alpha \\
&\ll P^2 \,T(P)^{-1}(P^{1-t})^{s-2-u} P^{u-k}\\
&\ll P^{s-k-(s-2-u)t}\,T(P)^{-1}\\
&= P^{s-k-\iota}\,T(P)^{-1},
\end{align*}
where $\iota=(s-2-u)t >0$.
\end{proof}

Consequently, we can add in the contribution from the trivial arcs to show that we have the required estimate on $\overline{\mathfrak{B}}$.
\begin{lemma} \label{lemmaBbar}
We have
\begin{equation*}
\int_{\overline{\mathfrak{B}}}\left|f_{\theta_1}(\alpha)f_{\theta_2}(\alpha)f_{\theta_3}(\alpha)^{s-2} K_\pm(\alpha)\right|\,d\alpha = o(P^{s-k}).
\end{equation*}
\end{lemma}
\begin{proof}
Combining Lemma \ref{slemma} with (\ref{K+-bound}) and (\ref{mesB}), we find that
\begin{align*}
\int_{\overline{\mathfrak{B}}}&\left|f_{\theta_1}(\alpha)f_{\theta_2}(\alpha)f_{\theta_3}(\alpha)^{s-2}K_\pm(\alpha)\right|\,d\alpha\\
&\ll \sum_{v=0}^{\infty} \int_{\overline{\mathfrak{B}}_{v+P^{\xi-k}}}\left|f_{\theta_1}(\alpha)f_{\theta_2}(\alpha)f_{\theta_3}(\alpha)^{s-2}K_\pm(\alpha)\right|\,d\alpha\\
&\ll(T(P)-P^{\xi-k})P^{s-k-\iota}\,T(P)^{-1}+\sum_{v=-1}^{\infty} \frac{L(P)}{(v+T(P))^2} P^2 (P^{1-t})^{s-2}P^{1/2-k}.
\end{align*}
Since $1/2<t(s-2)$ by our choice of $t$, we conclude that
\begin{align*}
\int_{\overline{\mathfrak{B}}}\left|f_{\theta_1}(\alpha)f_{\theta_2}(\alpha)f_{\theta_3}(\alpha)^{s-2}K_\pm(\alpha)\right|\,d\alpha&\ll P^{s-k-\iota}+\frac{L(P)}{T(P)-1}P^{s-k+1/2-t(s-2)}\\
&=o(P^{s-k}).
\end{align*}
\end{proof}

On $\mathfrak{B}$, we use a recent result of Baker which improves on an earlier result of Wooley. Firstly, we define $\alpha_0,\dotsc,\alpha_k$ by
\begin{equation} \label{obelisk}
\alpha(x-\theta_3)^k=\sum_{i=0}^k \alpha_ix^i,
\end{equation}
and note that $\alpha_k=\alpha$.

\begin{theorem}\label{ec1.6}
Let $k\geq 3$ be an integer, and let $t$ be a positive real number with $2k(k-1)t<1$. Let $\zeta$ be a sufficiently small positive real number. Suppose that P is sufficiently large, and that $\left|f_{\theta_3}(\alpha)\right|\geq P^{1-t}$. Then there exist integers $q,a_1,\dotsc,a_k$ such that $1\leq q\leq P^{1-\zeta}$ and $\left|q\alpha_j-a_j\right|\leq P^{1-j-\zeta}$ for $1\leq j\leq k$. These integers also satisfy $(q,a_1,\dotsc,a_k)=1$ and $(q,a_2,\dotsc,a_k)\leq 2k^2$.

\end{theorem}
\begin{proof}
This is the case $A=P^{1-t}$ of \cite[Theorem 4]{Bakersmallfrac}, which in itself is an improvement of \cite[Theorem 1.6]{effcong} in the light of \cite{BDG}. The direct conclusion is that for small $\lambda$, we have $1\leq q\leq P^{\lambda+kt}$ and $\left|q\alpha_j-a_j\right|\leq P^{-j+\lambda+kt}$ for $1\leq j\leq k$; by choosing $\lambda$ such that $2(k-1)(\lambda+t)<1$ and $0<\zeta<1-\lambda-kt$, we reach the conclusion given above.
It is also possible to extract from the proof of this result in \cite{Bakersmallfrac} that the greatest common divisor $d=(q,a_2,\dotsc,a_k)$ satisfies $d\leq 2k^2$. Restricting to $(q,a_1,\dotsc,a_k)=1$ can only reduce the values of $q$ and $d$, so nothing is lost by doing so.
\end{proof}

We introduce some more notation. For integers $q,a_1,\dotsc,a_k$, write
\begin{equation*}
S(q,\mathbf{a})=\sum_{x=1}^q e\Big(\frac{a_kx^k+\dotsc+a_1x}{q}\Big),
\end{equation*}
and for real numbers $\beta_1,\dotsc,\beta_k$, write
\begin{equation*}
I(\bm{\beta})=\int_0^Pe(\beta_ky^k+\dotsc+\beta_1y)\,dy.
\end{equation*}

We will use Theorem \ref{ec1.6} in conjunction with the following lemma.
\begin{lemma} \label{Bakerlemma}
Let $k\geq 2$. Let $f(x)=\alpha_kx^k+\dotsc+\alpha_1x$, and suppose that there are integers $q,a_1,\dotsc,a_k$ such that
\begin{equation*}
\left|q\alpha_j-a_j\right|\leq (2k^2)^{-1}P^{1-j},\quad(1\leq j\leq k).
\end{equation*}
Writing
\begin{equation*}
d=(q,a_2,\dotsc,a_k),
\end{equation*}
and
\begin{equation*}
\beta_j=\alpha_j-\frac{a_j}{q},\quad(1\leq j\leq k),
\end{equation*}
we have
\begin{equation*}
\sum_{x=1}^{P}e(f(x))=q^{-1}S(q,\mathbf{a})I(\bm{\beta})+O(q^{1-1/k+\epsilon}d^{1/k}).
\end{equation*}
\end{lemma}
\begin{proof}
This is \cite[Lemma 4.4]{baker}.
\end{proof}

By the definition of $\mathfrak{B}$, we have met the conditions of Theorem \ref{ec1.6} for $\alpha\in\mathfrak{B}$. Fixing a sufficiently small $\zeta>0$, and a choice of $\lambda$ with $2(k-1)(\lambda+t)<1$ and $0<\zeta<1-\lambda-kt$, we may let $q(\alpha), a_1(\alpha),\dotsc,a_k(\alpha)$ be integers meeting the conditions given in the conclusion of that theorem, namely that $1\leq q(\alpha)\leq P^{1-\zeta}$ and $\left|q(\alpha)\alpha_j-a_j(\alpha)\right|\leq P^{1-j-\zeta}$ for $1\leq j\leq k$. The narrow width of this permissible range for $\left|q(\alpha)\alpha_j-a_j(\alpha)\right|$ and the coprimality condition ensure that $q(\alpha)$ and $\mathbf{a}(\alpha)$ are well-defined. Let
\begin{equation*}
\beta_j(\alpha)=\alpha_j-\frac{a_j(\alpha)}{q(\alpha)},\quad(1\leq j\leq k),
\end{equation*}
and
\begin{equation*}
d(\alpha)=(q(\alpha),a_2(\alpha),\dotsc,a_k(\alpha))\ll 1.
\end{equation*}

Note that for sufficiently large $P$, we have $P^{-\zeta}\leq (2k^2)^{-1}$.
Recalling (\ref{obelisk}), we apply Lemma \ref{Bakerlemma} to conclude that
\begin{align*}
f_{\theta_3}(\alpha)&=\sum_{x=1}^{P}e(\alpha(x-\theta_3)^k)\\
&= \sum_{x=1}^{P}e(\alpha_kx^k+\dotsc+\alpha_1x+\alpha_0)\\
&\ll q(\alpha)^{-1}S(q(\alpha),\mathbf{a}(\alpha))I(\bm{\beta}(\alpha))+q(\alpha)^{1-1/k+\epsilon}.
\end{align*}

We now use \cite[Theorems 7.1 and 7.3]{vaughan} to provide estimates for $S(q(\alpha),\mathbf{a}(\alpha))$ and $I(\bm{\beta}(\alpha))$. We have
\begin{equation*}
S(q(\alpha),\mathbf{a}(\alpha)) \ll q(\alpha)^{1-1/k+\epsilon},
\end{equation*}
and
\begin{equation*}
I(\bm{\beta}(\alpha))\ll P(1+\left|\beta_1(\alpha)\right|P+\dotsc+\left|\beta_k(\alpha)\right|P^k)^{-1/k}.
\end{equation*}
Hence we see that
\begin{align}
f_{\theta_3}(\alpha)&\ll q(\alpha)^{-1/k+\epsilon}P(1+\dotsc+\left|\beta_k(\alpha)\right|P^k)^{-1/k}+q(\alpha)^{1-1/k+\epsilon}\nonumber \\
&\ll q(\alpha)^{-1/k+\epsilon}P(1+\left|\beta_k(\alpha)\right|P^k)^{-1/k}. \label{eq: betaest}
\end{align}
We now use this result to bound the integral that we are interested in.

\begin{lemma} \label{otherone}
For $u>2k$, we have
\begin{equation*}
\int_{\mathfrak{B}_v}\left|f_{\theta_3}(\alpha)\right|^u\,d\alpha \ll P^{u-k}.
\end{equation*}
\end{lemma}
\begin{proof}
By the above definitions, we note that if $q(\alpha)=q(\alpha')$,\\ $a_k(\alpha)=a_k(\alpha')$ and $\beta_k(\alpha)=\beta_k(\alpha')$, then in fact $\alpha=\alpha'$. Using (\ref{eq: betaest}), we therefore have
\begin{align*}
\int_{\mathfrak{B}_v}\left|f_{\theta_3}(\alpha)\right|^u\,d\alpha &\ll \int_{\mathfrak{B}_v}(q(\alpha)^{-1/k+\epsilon}P(1+\left|\beta_k(\alpha)\right|P^k)^{-1/k})^u\,d\alpha\\
&\ll P^u \sum_{1\leq q\leq P^{1-\zeta}} \sum_{a_k=1}^q q^{-u/k+\epsilon} \int_{\left|\beta_k\right|\leq P^{1-k-\zeta}}(1+\left|\beta_k\right|P^k)^{-u/k}\,d\beta_k\\
&\ll P^u J \sum_{1\leq q\leq P^{1-\zeta}}q^{1-u/k+\epsilon},
\end{align*}
where, just as in \cite[Corollary 2.4]{chowWP}, we have
\begin{equation*}
J=\int_0^\infty(1+\beta P^k)^{-u/k}\,d\beta \ll P^{-k}.
\end{equation*}
Consequently, since $u/k>2$ and $\epsilon$ is small, we see that
\begin{align*}
\int_{\mathfrak{B}_v}\left|f_{\theta_3}(\alpha)\right|^u\,d\alpha &\ll P^{u-k} \sum_{1\leq q\leq P^{1-\zeta}}q^{1-u/k+\epsilon}\\
&\ll P^{u-k}.
\end{align*}
\end{proof}

\begin{lemma} \label{Bvlemma}
For $v$ with $\mathfrak{B}_v \subset \mathfrak{m}$, we have
\begin{align*}
\int_{\mathfrak{B}_v}\left|f_{\theta_1}(\alpha)f_{\theta_2}(\alpha)f_{\theta_3}(\alpha)^{s-2}\right|\,d\alpha &\ll P^{s-k}\,T(P)^{-1}.
\end{align*}
\end{lemma}
\begin{proof}
Using (\ref{TPbound}), we have
\begin{align*}
\int_{\mathfrak{B}_v}\left|f_{\theta_1}(\alpha)f_{\theta_2}(\alpha)f_{\theta_3}(\alpha)^{s-2}\right|\,d\alpha &\ll \sup_{\alpha\in\mathfrak{B}_v}\left|f_{\theta_1}(\alpha)f_{\theta_2}(\alpha)\right| \int_{\mathfrak{B}_v}\left|f_{\theta_3}(\alpha)\right|^{s-2}\,d\alpha\\
&\ll P^2 \,T(P)^{-1}\int_{\mathfrak{B}_v}\left|f_{\theta_3}(\alpha)\right|^{s-2}\,d\alpha.
\end{align*}
Since we may suppose that $s>2k+2$, we apply Lemma \ref{otherone} with $u=s-2$ to obtain
\begin{align*}
\int_{\mathfrak{B}_v}\left|f_{\theta_1}(\alpha)f_{\theta_2}(\alpha)f_{\theta_3}(\alpha)^{s-2}\right|\,d\alpha &\ll P^2 \,T(P)^{-1}P^{s-2-k}\\
&\ll P^{s-k}\,T(P)^{-1}.
\end{align*}
\end{proof}

We now combine the minor and trivial arc estimates to deduce the required result for the whole of $\mathfrak{B}$.
\begin{lemma} \label{lemmaB}
We have
\begin{align*}
\int_{\mathfrak{B}}\left|f_{\theta_1}(\alpha)f_{\theta_2}(\alpha)f_{\theta_3}(\alpha)^{s-2} K_\pm(\alpha)\right|\,d\alpha &= o(P^{s-k}).
\end{align*}
\end{lemma}
\begin{proof}
As in Lemma \ref{lemmaBbar}, we split the integral over $\mathfrak{B}$ into integrals over $\mathfrak{B}_v$, distinguishing between those intervals contained in $\mathfrak{m}$, and those contained in (or intersecting) $\mathfrak{t}$. Using  Lemma \ref{Bvlemma} and (\ref{K+-bound}), and writing $\omega=P^{\xi-k}$ and $z=T(P)-\omega-1$ for brevity, we have
\begin{align*}
\sum_{0\leq v\leq z}\int_{\mathfrak{B}_{v+\omega}}&\left|f_{\theta_1}(\alpha)f_{\theta_2}(\alpha)f_{\theta_3}(\alpha)^{s-2}K_\pm(\alpha)\right|\,d\alpha\\
&\ll P^{s-k}\,T(P)^{-1}+\sum_{1\leq v\leq z}\frac{L(P)}{(v+\omega)^2} \int_{\mathfrak{B}_{v+\omega}}\left|f_{\theta_1}(\alpha)f_{\theta_2}(\alpha)f_{\theta_3}(\alpha)^{s-2}\right|\,d\alpha\\
&\ll P^{s-k}\,T(P)^{-1}+ P^{s-k}\,\frac{L(P)}{T(P)}\sum_{1\leq v\leq z}\frac{1}{(v+\omega)^2} \\
&=o(P^{s-k}),
\end{align*}
while, by Lemma \ref{otherone},
\begin{align*}
\sum_{v=-1}^{\infty} \int_{\mathfrak{B}_{v+T(P)}}\left|f_{\theta_1}(\alpha)f_{\theta_2}(\alpha)f_{\theta_3}(\alpha)^{s-2}K_\pm(\alpha)\right|\,d\alpha
&\ll \sum_{v=-1}^{\infty} \frac{L(P)}{(v+T(P))^2} P^2 P^{s-2-k}\\
&\ll \frac{L(P)}{T(P)-1}P^{s-k}\\
&=o(P^{s-k}).
\end{align*}
Combining the above sums, we achieve the stated result for the whole of $\mathfrak{B}$.
\end{proof}

We now summarise our conclusion for the whole of the Davenport--Heilbronn minor and trivial arcs in another lemma.
\begin{lemma} \label{minortotal}
For any natural number $s\geq s_0(k)$, we have
\begin{equation*}
\int_{\mathfrak{m}\cup\mathfrak{t}}f_{\bm{\theta}}(\alpha) e(-\tau\alpha)K_\pm(\alpha)\,d\alpha = o(P^{s-k}).
\end{equation*}
Assuming the shifted Hua's lemma, the same result holds whenever $s\geq s_1(k)$.
\end{lemma}
\begin{proof}
By symmetry, the results of this section hold equally well when $\theta_3$ is replaced by any other $\theta_i$ with $i\geq 4$. Consequently, applying H\"older's inequality, we see that
\begin{align*}
\int_{\mathfrak{N}}f_{\bm{\theta}}(\alpha) e(-\tau\alpha)K_\pm(\alpha)\,d\alpha &\ll \prod_{i=3}^s\Bigg(\int_{\mathfrak{N}}\left|f_{\theta_1}(\alpha)f_{\theta_2}(\alpha)f_{\theta_i}(\alpha)^{s-2}K_\pm(\alpha)\right|\,d\alpha\Bigg)^{1/(s-2)}\\
&=o(P^{s-k}),
\end{align*}
by Lemmata \ref{lemmaBbar} and \ref{lemmaB}. By Corollary \ref{minormixed}, and using the dissection $\mathfrak{m}\cup\mathfrak{t}=\mathfrak{N}\cup\mathfrak{n}$, we achieve the desired conclusion.
\end{proof}

\section{The major arc} \label{majorarc}

On the major arc
\begin{equation*}
\mathfrak{M}=\{\alpha\in\mathbb{R}:\left|\alpha\right|< P^{\xi-k}\},
\end{equation*}
we use a result of Chow, noting that it requires only that the number of variables be greater than $k$.

\begin{lemma} \label{chowmajor}
We have
\begin{equation*}
\int_{\mathfrak{M}}f_{\bm{\theta}}(\alpha) e(-\tau\alpha)K_\pm(\alpha)\,d\alpha=2\eta\Gamma(1+1/k)^s\Gamma(s/k)^{-1}P^{s-k}+o(P^{s-k}).
\end{equation*}
\end{lemma}
\begin{proof}
This is \cite[equation (3.28)]{chowWP}.
\end{proof}

Combining Lemmata \ref{minortotal} and \ref{chowmajor} with (\ref{DHdissect}), we obtain the conclusion of Theorem \ref{mainThm}. Assuming the shifted Hua's lemma, we would achieve the same result whenever $s\geq s_1(k)$. In particular, this would provide a further improvement when $k=10$ or $k\geq 12$.

\newcommand{\noop}[1]{}


\begin{thebibliography}{10}

\bibitem{baker}
R.~C. Baker.
\newblock \emph{Diophantine inequalities}.
\newblock The Clarendon Press, Oxford University Press, New York, 1986.

\bibitem{Bakersmallfrac}
R.~C. Baker.
\newblock Small fractional parts of polynomials.
\newblock \emph{Funct. Approx. Comment. Math.} \textbf{55} (2016), no.~1,
  131--137.

\bibitem{bourgainexpl}
J.~Bourgain.
\newblock On the {V}inogradov mean value, \noop{2016}ar{X}iv:1601.08173.

\bibitem{BDG}
J.~Bourgain, C.~Demeter, and L.~Guth.
\newblock Proof of the main conjecture in {V}inogradov's mean value theorem for
  degrees higher than three.
\newblock \emph{Ann. of Math. (2)} \textbf{184} (2016), no.~2, 633--682.

\bibitem{cubes}
S.~Chow.
\newblock Sums of cubes with shifts.
\newblock \emph{J. Lond. Math. Soc. (2)} \textbf{91} (2015), no.~2, 343--366.

\bibitem{chowWP}
S.~Chow.
\newblock Waring's problem with shifts.
\newblock \emph{Mathematika} \textbf{62} (2016), no.~1, 13--46.

\bibitem{davenport}
H.~Davenport.
\newblock \emph{Analytic methods for {D}iophantine equations and {D}iophantine
  inequalities}.
\newblock Second edition. Cambridge University Press, Cambridge, 2005.

\bibitem{davheil}
H.~Davenport and H.~Heilbronn.
\newblock On indefinite quadratic forms in five variables.
\newblock \emph{J. London Math. Soc.} \textbf{21} (1946), 185--193.

\bibitem{freeman}
D.~E. Freeman.
\newblock Asymptotic lower bounds and formulas for {D}iophantine inequalities.
\newblock In \emph{Number theory for the millennium, II (Urbana, IL, 2000)},
  57--74. A K Peters, Natick, MA, 2002.

\bibitem{excepsets}
S.~T. Parsell and T.~D. Wooley.
\newblock Exceptional sets for {D}iophantine inequalities.
\newblock \emph{Int. Math. Res. Not. IMRN} \textbf{2014}, no.~14, 3919--3974.

\bibitem{vaughan}
R.~C. Vaughan.
\newblock \emph{The {H}ardy-{L}ittlewood method}.
\newblock Second edition. Cambridge University Press, Cambridge, 1997.

\bibitem{asympform}
T.~D. Wooley.
\newblock The asymptotic formula in {W}aring's problem.
\newblock \emph{Int. Math. Res. Not. IMRN} \textbf{2012}, no.~7, 1485--1504.

\bibitem{effcong}
T.~D. Wooley.
\newblock Vinogradov's mean value theorem via efficient congruencing.
\newblock \emph{Ann. of Math. (2)} \textbf{175} (2012), no.~3, 1575--1627.

\bibitem{wooleyk3}
T.~D. Wooley.
\newblock The cubic case of the main conjecture in {V}inogradov's mean value
  theorem.
\newblock \emph{Adv. Math.} \textbf{294} (2016), 532--561.

\end{thebibliography}
\end{document}